\documentclass[12pt]{article}
\usepackage{amsmath,fullpage,amsthm}
\usepackage{amscd}
\usepackage{amsthm,mathtools} \usepackage{amssymb}
\usepackage{latexsym}
\usepackage{eufrak}
\usepackage{euscript}
\usepackage[width=17.5cm,height=24cm]{geometry}
\usepackage{epsfig}
\usepackage{tikz}
\usepackage{graphics}
\usepackage{array}
\usepackage{enumerate} \usepackage{authblk}

\theoremstyle{theorem}
\newtheorem{theorem}{Theorem}[section]
\theoremstyle{corollary}

\theoremstyle{lemma}
\newtheorem{lemma}{Lemma}[section]
\theoremstyle{definition}

\theoremstyle{proof}
\theoremstyle{remark}
\newtheorem{remark}{Remark}[section]
\theoremstyle{example}

\theoremstyle{observation}

\begin{document}
  \setcounter{Maxaffil}{2}
\title{Further study on forbidden subgraphs of power graph}
   \author[ a]{Santanu Mandal\thanks{santanu.vumath@gmail.com}}
   \author[b ]{Pallabi Manna\thanks{mannapallabimath001@gmail.com}}
   \affil[a ]{School of Computing Science and Engineering,}
   \affil[ ]{Vellore Institute of Technology,}
   \affil[ ]{Bhopal - 466114, India}
   \affil[b ]{Harish Chandra Research Institute,}
   \affil[ ]{Prayagraj - 211019, India}
   \maketitle

\begin{abstract}
The undirected power graph (or simply power graph) of a group $G$, denoted
by $P(G)$, is a graph whose vertices are the elements of the group $G$, in
which two vertices $u$ and $v$ are adjacent if and only
if either $u=v^m$ or $v=u^n$ for some positive integers $m$, $n$.
Forbidden subgraph has a significant role in graph theory. In our previous work \cite{cmm}, we consider five important classes of forbidden subgraphs of power graph which include perfect graphs, cographs, chordal graphs, split graphs and threshold graphs. In this communication, we go even further in that way. This study, inspired by the articles \cite{celmmp,dong,ck},  examines additional $4$ significant forbidden classes, including chain graphs, diamond-free graphs, $\{P_{5}, \overline{P_{5}}\}$-free graphs and $\{P_{2}\cup P_{3}, \overline{P_{2}\cup P_{3}}\}$-free graph. The finite groups whose power graphs are chain graphs, diamond-free graphs, and $\{P_{2}\cup P_{3}, \overline{P_{2}\cup P_{3}}\}$-free graphs have been successfully identified in this work. In case of $\{P_{5}, \overline{P_{5}}\}$-free graphs, we completely determine all the nilpotent groups, direct product of two groups, finite simple groups  whose power graph is $\{P_{5}, \overline{P_{5}}\}$-free. 
\end{abstract}

   \section{Introduction}
Graphs defined on various algebraic structures like groups, rings, vector spaces  become very popular among the researchers from last few decades. Power graph is one such major graph representation of semigroups, groups. In 2002,  Kelarev and Quinn introduced the directed power graph of semigroups (see \cite{Kelarev}). The directed power graph of a semigroup $S$, denoted by $\overrightarrow{P}(S)$), is a graph whose vertex set is $S$ and there is an arc $u\rightarrow v$ (where, $u\neq v$) if $v=u^{m}$ for some positive integer $m$. The corresponding underlying graph is called the undirected power graph, which is denoted by $P(S)$. Chakrabarty et al. \cite{Chakrabarty} introduced the idea of this graph in 2009. Throughout the paper we consider the power graph means the undirected power graph of finite group. In particular if we remove the identity element of the group $G$ from its original power graph $P(G)$ then the remaining graph is known as the reduced power graph or a proper power graph which is denoted by the symbol $P^*(G)$. In \cite{Doostabadi}, the authors introduced the proper power graph of a group. For more existing results regarding power graph we refer the articles \cite{Abawajy, Chakrabarty}. \\
Forbidden subgraph has an extensive role in graph theory. There are several graph classes that can be represent in terms of forbidden subgraphs. In \cite{Br}, Brandst et al. discussed about various graph classes which is represented by the forbidden subgraphs. A graph is said to be $H$-free if it does not contain $H$ as its induced subgraph. In graph theory there are plenty of research articles are available in which the researchers deal with any NP-complete problem or any structural properties of a particular type of forbidden subgraph class. In this direction we refer few of such  articles, namely \cite{celmmp, dong, am, ck, ck1, ck2, hk, ckmm}.\\
In our previous work \cite{cmm}, we consider several classes of forbidden subgraphs like perfect graphs, cographs, chordal graphs, split graphs and threshold graphs. Moreover in \cite{Mehatari}, we discussed about the direct product of two groups, simple groups of Lie type whose power graph is a cograph.\\
Motivated by the above mentioned articles we consider the graph classes like chain graph, $\{P_{5}, \overline{P_{5}}\}$-free graph, $\{P_{2}\cup P_{3}, \overline{P_{2}\cup P_{3}}\}$-free graph and diamond free graph in case of power graph of finite groups. These graph classes are one of the important forbidden graph class because the class $\{P_{5}, \overline{P_{5}}\}$-free is a very large class of graph that contains cographs, even-hole (of length more than $4$) free graphs, odd-hole (of length more than $5$) free graphs, threshold graphs, complete graphs etc. On the other hand, the classes like complete graphs, complete bipartite graphs, cluster graphs, perfect graphs, even-hole (of length more than $6$) etc. are the subclass of a $\{P_{2}\cup P_{3}, \overline{P_{2}\cup P_{3}}\}$-free graph. Furthermore, this approach provides the benefits for handling several open NP-complete problems in the context of power graphs. The problems with power graphs of any arbitrary finite group, such as clique number, chromatic number, Hamiltonicity, clique-width, graph partition problem, etc. that are difficult to solve, in such cases, by taking into consideration some forbidden subgraph classes of power graphs, we can at least partially come to a conclusion regarding these problems.\\
The paper is organized following this manner: in section $2$ we recall some basic definitions, theorems which we use in this study and the notainal conventions of this paper. In section $3$ we conclude that the proper power graph of a finite group $G$ is a chain graph if and only if $G$ is either a) $C_{3}$ or b) a $2$-group of exponent $2$ or c) a EPO group $C_{3}\rtimes P$, where $P$ is a non-cyclic $2$-group of exponent $2$ or d) $S_{3}$. A group is called EPPO if every non-identity elements are of prime power order; whereas if every non-identity element of a group are of prime order then it is called an EPO group.\\
Section $4$ is devoted to  $\{P_{5}, \overline{P_{5}}\}$-free graph. Here we conclude the necessary and sufficient condition of a $\{P_{5}, \overline{P_{5}}\}$-free graph in case of nilpotent groups and direct product of two groups and obtain the following results.
\begin{theorem}
\label{th_p5_1}
Let $G$ be a finite nilpotent group. Then $P(G)$ is $\{P_{5}, \overline{P_{5}}\}$-free if and only if $G$ is either a) a $p$ group or b) a cyclic group $C_{p^{a}q}$, where $p, q$ are distinct primes and $a \geq 1$.
\end{theorem}
\begin{theorem}
\label{th_p5_direct_product}
Let $G, H$ be two finite groups. Then $P(G\times H)$ is $\{P_{5}, \overline{P_{5}}\}$-free if and only if $G, H$ take one of the following forms:\\
a) both $G, H$ are power of same prime;\\
b) one of $G, H$ is the cyclic group $C_{p^k}$ and the other one is $C_{q}$, where $p, q$ are distinct primes;\\
c) one of $G, H$ is a cyclic group $C_{q^m}$ then we have the other is the group (i) or (ii).\\
i) $C_{p^r}\rtimes Q$ ($r \geq 1$) if $m=1$.\\
ii) $C_{p}\times Q$ if $m>1$.\\
Provided $p, q$ are distinct prime divisors of $o(G\times H)$ and $Q$ is Sylow $q$-subgroup of $H$.
\end{theorem}
We classify the low dimensional simple groups of Lie type whose power graph is  $\{P_{5}, \overline{P_{5}}\}$-free. We obtain the following result:
\begin{theorem}
Let $G$ be a finite simple group of Lie type except the Ree group ${}^{2}{G_{2}}(q)$ (where, $q=3^{2e+1}$). Then $P(G)$ is $\{P_{5}, \overline{P_{5}}\}$-free if and only if either of the followings hold:\\
I) $G\cong A_{n}$ with $n \leq 6$;\\
II) $G\cong PSL(2, q)$ such that conditions a) or b) occurs:
a) the numbers $(q\pm 1)/2$ are either a prime or product of some prime and a prime power if $q$ odd;\\
b) $q\pm 1$ are either a prime or product of some prime and a prime power if $q$ even;\\
III) $G={}^{2}{B_{2}}(q)=Sz(q)$, where $q=2^{2e+1}$ with the numbers $q-1, q\pm\sqrt{2q}+1$ are  either a prime or product of some prime and a prime power;\\
IV) $G\cong PSL(3, 4)$.
\end{theorem}
Additionally, we show that there is no sporadic simple groups whose power graph is  $\{P_{5}, \overline{P_{5}}\}$-free. \\
In section $5$ we find a necessay and sufficient condition for a nilpotent group as well as a non-nilpotent group whose power graph is $\{P_{2}\cup P_{3}, \overline{P_{2}\cup P_{3}}\}$-free. And the final section i.e., section $6$ determines the finite groups having diamond-free power graph.

\section{Preliminaries}
We use the notations $K_{n}, P_{n}, C_{n}, 2K_{2}, \overline{\Gamma}, \Gamma_{1}\cup \Gamma_{2}$ to indicate a complete graph of order $n$, a path on $n$-vertices, a cycle of length $n$, the complement of $C_{4}$, complement of the graph $\Gamma$, disjoint unions of two graphs $\Gamma_{1}$ and $\Gamma_{2}$  respectively. From group theory we use the standard notations like $o(G), o(a), C_{n}, S_{n}, A_{n}, G\times H$ and $G\rtimes H$ to mean the order of the group $G$, the order of the element $a$ of a group, a cyclic group of order $n$, a symmetric group on $n$-symbols, an alternating group on $n$-symbols, the direct product of two groups $G, H$ and the semi-direct product of $G, H$. We use the same notation $C_{n}$ for both the cyclic group of order $n$ and a cycle of length. This will be clear from the context which we intend. The notation $\pi(G)$ stands for the set of all distinct prime divisors of $o(G)$, and $|\pi(G)|$ is the cardinality of $\pi(G)$.\\
We now want to recall the definition of nilpotent group. A group is nilpotent if it is the direct products of its Sylow subgroups.
Power graph has one important property that for a given group $G$, the power graph of any subgroup of $G$ is an induced subgraph of $P(G)$. This property helps us to determine a group whose power graph is whether lies in the classes of graph considered in this paper.\\
In \cite{cmm}, we completely characterized finite nilpotent power-cograph
groups. We proved the following theorem:
   \begin{theorem}[\cite{cmm}, Theorem 3.2]
\label{pre_th_1}
Let $G$ be a finite nilpotent group. Then $P(G)$ is a cograph if and only if
either $|G|$ is a prime power, or $G$ is cyclic of order $pq$ for distinct
primes $p$ and $q$.
\end{theorem}
For a given group $G$ its prime graph is the graph whose vertex set is the distinct prime divisors of $o(G)$ and there is an edge between any two distinct primes if $G$ has an element of order product of these two distinct primes. Earlier (see \cite{cmm}) we proved that:
\begin{theorem}
\label{pre_th_2}
Let $G$ be a group whose prime graph is a null graph. Then $P(G)$ is a
cograph.
\end{theorem}
Moreover if the prime graph of a group is a null graph (or in other words, the group is an EPPO group) then any two adjacent vertices must belong to the same cyclic subgroup of prime power order. So, in that case $P(G)$ contains neither an induced path of length $2$ and above nor any induced cycle of length more than $3$.

A graph is chordal if it contains no induced cycles of length greater than $3$. We recall a theorem from \cite{cmm} that determines the chordality of the power graph of a nilpotent group.
\begin{theorem}
\label{pre_th_3}
Let $G$ be a finite nilpotent group. Then
$P(G)$ is chordal if and only if $G$ is either a group of prime power order or $|G|$ has two prime divisors, one
of the two Sylow subgroups is cyclic, and the other has prime exponent.
\end{theorem}

\section{Chain Graph}
A graph is called chain graph if it forbids $\{C_{3}, C_{5}, 2K_{2}\}$. It is obvious that if we consider the power graph of any finite group $G$ then $P(G)$ is a chain graph if and only if $G$ is a $2$-group of exponent $2$. Thus, in this section we consider the proper power graph and determine the groups whose proper power graph is a chain graph.
\begin{theorem}
For any finite group $G$, $P^{*}(G)$ is a chain graph if and only if $G$ is either a) $C_{3}$ or b) a $2$-group of exponent $2$ or c) a EPO group $C_{3}\rtimes P$, where $P$ is a non-cyclic $2$-group of exponent $2$ or d) $S_{3}$.
\end{theorem}

\begin{proof}
Let, $P^*(G)$ be a chain graph.\\
Since $P^{*}(G)$ is $C_{3}$-free, so it does not contain an element of order $\geq 4$. Clearly, $o(G)$ has at most two distinct prime divisors. Otherwise, $o(G)$ has at least one odd prime divisor say $p\geq 5$. Then there exists an element $a$ of order $p$ and $\{a, a^{2}, a^{-1}\}$ generate a $C_{3}$ in $P^{*}(G)$. \\
Now, if $G$ is a $p$-group then $G$ must be either a $2$-group of exponent $2$ or a $3$-group of exponent $3$. But if $G$ is a $3$-group of exponent $3$ then $P^{*}(G)$ is the disjoint union of multiple copies of $K_{2}$. So $P^*(G)$ contains $2K_{2}$ unless $G$ will be $C_{3}$. Therefore $G$ is either a) $C_{3}$ or b) a $2$-group of exponent $2$.\\
Next consider $o(G)$ has two distinct prime divisor. Since, $G$ cannot have any element of order $\geq 4$ so $\pi(G)=\{2, 3\}$ and $G$ is an EPO-group. Again, we observe that the Sylow $3$-subgroup must be normal and cyclic; elsewhere there exist two elements, say $a, b$, of order $3$ in $G$ such that the pairs $\{a, a^2\}, \{b, b^2\}$ form $2K_{2}$. Thus $G\cong C_{3}\rtimes P$ with $P$ is a $2$-group of exponent $2$.\\
In particular, if $P$ is cyclic then $G \cong S_{3}$.\\
Converse:\\
a) If $G\cong C_{3}$ then $P^*(C_{3})$ is a complete graph $K_{2}$. Thus, $P^*(G)$ is a chain graph.\\
b) Let $G$ be a $2$-group of exponent $2$ then $P^{*}(G)$ is the disjoint union of isoated vertices. Thus, $P^{*}(G)$ is a chain graph.\\
c) Let $G$ be a EPO group $C_{3}\rtimes P$, where $P$ is a $2$-group of exponent $2$. Clearly, $P^*(G)$ is the disjoint union of $K_{2}$  and some isolated vertices. This implies that $P^*(G)$ is $\{C_{3}, C_{5}, 2K_{2}\}$-free. Hence $P^*(G)$ is a chain graph.\\
d) If $G\cong S_{3}$ then $P^*(G)$ is $K_{2}\cup 3K_{1}$. So it is a chain graph.
\end{proof}
\section{$\{P_{5}, \overline{P_{5}}\}$-free}
 In this section we consider the finite nilpotent groups, simple groups of Lie type and sporadic simple groups, and explore those groups whose power graph is $\{P_{5}, \overline{P_{5}}\}$-free. Moreover, we also find the structures of two finite groups $G$ and $H$ such that $P(G\times H)$ is $\{P_{5}, \overline{P_{5}}\}$-free.
\subsection{Nilpotent group, Direct product of two groups}
\begin{theorem}
\label{th_p5}
Let $G$ be a finite nilpotent group. Then $P(G)$ is $P_{5}$-free if and only if $G$ is either a $p$-group or a cyclic group $C_{p^{a}q}$, where $p, q$ are distinct primes and $a\geq 1$.
\end{theorem}

\begin{proof}
Let $G$ be a finite nilpotent group such that $P(G)$ is $P_{5}$-free.\\
\textbf{\underline{Claim 1.}} $o(G)$ has at most two distinct prime divisors.\\
\textbf{Proof of Claim 1:} Suppose, $p, q, r$ are $3$ distinct primes divide $o(G)$. Let $a, b, c$ be the elements of order $p, q, r$ respectively. Then $P(G)$ contains a path $a\sim ab \sim b \sim bc\sim c$. \\
According to the above claim we have either $G$ is a $p$-group or $o(G)=p^{a}q^{b}$, where $p, q$ are distinct primes and $a, b \geq 1$. We now consider the following cases.\\
\textbf{Case 1.} Let $G$ be a $p$-group.\\
By Theorem \ref{pre_th_1}, $P(G)$ is a cograph implies $P(G)$ is $P_{5}$-free.\\
\textbf{Case 2.} Let $o(G)$ has two distinct prime divisors say $p, q$.\\
Suppose, $o(G)=p^{a}q^{b}$ with $a, b \geq 1$. Let $P, Q$ be the Sylow $p$- and Sylow $q$-subgroups of $G$.\\
\textbf{\underline{Claim 2.}} We claim that both  Sylow subgroups must be cyclic.\\
For the sake of contradiction, let the Sylow $p$-subgroup $P$ be non-cyclic. Then there exist elements say $a, b$ of order $p$ such that $a \nsim b$ in $P(G)$. In that case, $P(G)$ contains a path $a\sim ac\sim c\sim bc \sim b$, where $c \in Q$. Therefore, $G \cong C_{p^{a}q^{b}}$ with $a, b \geq 1$. \\
If both $a, b >1$ then again the path $a \sim b \sim c \sim d \sim e$ is contained in $P(G)$, where $o(a)=p^{a}, o(b)=p, o(c)=pq, o(d)=q, o(e)=q^{b}$. Thus one of $a$ or $b$ must be $1$ and hence $G$ is the cyclic group $C_{p^{a}q}$ with $a\geq 1$.\\
Converse part is obvious.
\end{proof}

\begin{theorem}
\label{th_p5_1}
Let $G$ be a finite nilpotent group. Then $P(G)$ is $\{P_{5}, \overline{P_{5}}\}$-free if and only if $G$ is either i) a $p$ group or ii) a cyclic group $C_{p^{a}q}$, where $p, q$ are distinct primes and $a \geq 1$.
\end{theorem}

\begin{proof}
Let $G$ be a finite nilpotent group such that $P(G)$ is $\{P_{5}, \overline{P_{5}}\}$-free. Since, $P(G)$ is $P_{5}$-free so by Theorem \ref{th_p5} $G$ must be the groups i) or ii).\\
\underline{Converse Part:} \\
i) Let $G$ be a $p$-group. If $P(G)$ contains $\overline{P_{5}}$ then $P(G)$ must comprises a $4$-vertex induced path. This contradicts the Theorem \ref{pre_th_1}, that is $P(G)$ is a cograph. \\
ii) On the other hand, let $G$ be the cyclic group $C_{p^{a}q}$ ($a, b \geq 1$). Suppose $P(G)$ has induced subgraph $\overline{P_{5}}$. We choose $4$ consecutive vertices of orders $p^{a}, p, pq, q$. Then the fifth vertex must be of order $p^{i}q$, which is adjacent to both the second and third vertices. This leads to a contradiction that $\overline{P_{5}}$ is an induced subgraph. \\
Thus, in any cases $P(G)$ is $\{P_{5}, \overline{P_{5}}\}$-free.
\end{proof}

\begin{theorem}
\label{th_p5_direct_product}
Let $G, H$ be two finite groups. Then $P(G\times H)$ is $\{P_{5}, \overline{P_{5}}\}$-free if and only if $G, H$ take one of the following forms:\\
a) both $G, H$ are power of same prime;\\
b) one of $G, H$ is the cyclic group $C_{p^k}$ and the other one is $C_{q}$, where $p, q$ are distinct primes;\\
c) one of $G, H$ is a cyclic group $C_{q^m}$ then we have the other is the group either (i) or (ii).\\
i) $C_{p^r}\rtimes Q$ ($r \geq 1$) if $m=1$.\\
ii) $C_{p}\times Q$ if $m>1$.\\
Provided $p, q$ are distinct prime divisors of $o(G\times H)$ and $Q$ is Sylow $q$-subgroup of $H$.
\end{theorem}

\begin{proof}
Suppose $P(G\times H)$ is $\{P_{5}, \overline{P_{5}}\}$-free. \\
First observe that $o(G\times H)$ has at most two distinct prime divisors. Elsewhere there exists $3$ primes say $p, q, r$  such that $p, q|o(G)$ and $r|o(H)$. Then $P(G\times H)$ contains a path $a\sim ac\sim  c \sim bc \sim b$, where $o(a)=p, o(b)=q$ and $o(c)=r$. \\
If $G$ and $H$ are both power of same prime then $G\times H$ is a $p$-group. So, there is nothing to prove (see Theorem \ref{th_p5_1}).\\
Let $o(G\times H)$ has precisely two distinct prime divisors say $p, q$. \\
If both $G, H$ are abelian then $G, H$ have the structure in b).\\
Let us assume that both $G, H$ can not be abelian.\\
\textbf{\underline{Claim:}} one of $o(G), o(H)$ must be of prime power order.\\
\textbf{Proof of the claim} Let $p, q$ be two primes such that $pq|o(G), o(H)$ both. Consider the Sylow $p$- and Sylow $q$-subgroups of both $G, H$ are $P_{G}, Q_{G}$ and $P_{H}, Q_{H}$ respectively. Now $P_{G}\times Q_{H}$ is nilpotent and $\{P_{5}, \overline{P_{5}}\}$-free. Without loss of generality let $P_{G}\cong C_{p^a}$  and $Q_{H}\cong C_{q}$ ( by Theorem \ref{th_p5_1}). Similarly, $Q_{G}\times P_{H}$ is nilpotent and $\{P_{5}, \overline{P_{5}}\}$-free implies $Q_{G}\cong C_{q^b}$ ($b\geq 1$) and $P_{H}\cong C_{p}$ or $Q_{G}\cong C_{q}$ and $P_{H}\cong C_{p^k}$ ($k\geq 1$). Also, none of $G, H$ contains any abelian subgroup say $M$ of order $pq$; otherwise $G\times H$ contains one of the four nilpotent subgroups  $P_{G}\times M$, $Q_{G}\times M$, $M\times P_{H}$, $M\times Q_{H}$, whose power graphs comprise a $P_{5}$ [see Theorem \ref{th_p5_1}]. \\
Since, $G,~H$ can not have abelian subgroup of order $pq$ so they are not nilpotent.  So, one of the two Sylow subgroups of them must be not normal. Without loss of generality, let the Sylow $q$-subgroups of $G$ be not normal. Then the Sylow $p$-subgroup $P_{H}$ must be normal as otherwise $P(G\times H)$ has $P_{5}$. Using the similar argument we can say that the Sylow $q$-subgroup of $H$ ($Q_{H}$) is not normal and the Sylow $p$-subgroup of $G$ ($P_{G}$) is normal (since both $G, H$ are not nilpotent). Hence $G\cong C_{p^a}\rtimes C_{q}$ and $H \cong C_{p^k}\rtimes C_{q}$ or $G \cong C_{p^a}\rtimes C_{q^b}$ and $H\cong C_{p}\rtimes C_{q}$. But in any of the cases we obtain that $G$ contains two elements of order $q$ (say, $a, b$) such that $a\nsim b$ in $P(G\times H)$ and $H$ contains an element of order $p$, say $c$, for which $P(G\times H)$ has a path $a\sim ac\sim c\sim bc\sim b$. \\
Therefore, one of $G, H$ must be the group of prime power order. Let $G$ be the group with $o(G)=q^{m}$ and $pq|o(H)$. \\
Clearly, $G$ must be cyclic elsewhere $G$ contains two elements of order $q$ that are non-adjacent in $P(G\times H)$ and $H$ has an element of order $p$ such that $P(G\times H)$ comprises a $5$-vertex induced path.\\
Let $P_{H}, Q$ be Sylow $p$- and Sylow $q$-subgroups of $H$. Now, $G\times P_{H}$ is nilpotent gives either $G\cong C_{q}$ and $P_{H}\cong C_{p^r}$ or $G\cong C_{q^m}$ and $P_{H}\cong C_{p}$. But clearly $H$ does not consist any abelian subgroup $K$ of order $pq$ as otherwise $G\times H$ contains a nilpotent subgroup $G\times K$ whose power graph has $P_{5}$ (by Theorem \ref{th_p5_1}). This implies that either $G\cong C_{q}$ and $H\cong C_{p^r}\rtimes Q$ or $G\cong C_{q^m}$ and $H\cong C_{p}\rtimes Q$, where $Q$ is the Sylow $q$-subgroup of $H$. Thus we get the structures of $G, H$ which take the form as given in c).\\
\underline{Converse Part:}\\
If $G$ and $H$ are either a) or b) then $P(G\times H)$ is $\{P_{5}, \overline{P_{5}}\}$-free by Theorem \ref{th_p5_1}.\\
Let $G$ and $H$ have the form as in c). Firstly, we prove that $P(G)$ is $P_{5}$-free. If possible let $P(G)$ contain a $5$-vertex induced path $P_{5}$. Let $x_{1}\rightarrow x_{2}\leftarrow x_{3}\rightarrow x_{4}\leftarrow x_{5}\rightarrow \cdots$ be such a path, where $x_{i}=(g_{i}, h_{i})$. \\
If $o(h_{3})=1$, then $x_{2}\sim x_{4}$ as $G$ is cyclic.\\
If $o(g_{3})=1$, then again $x_{2}\sim x_{4}$ since $H$ contains only elements of order power of either $p$ or $q$.\\
Let $o(g_{3})$ and $o(h_{3})$ both be a  power of $q$. Then we obtain that $x_{2}\sim x_{4}$. \\
Suppose $o(g_{3})$ is a power of $q$ and $o(h_{3})$ is a power of $p$. Let $g_{3}=a, h_{3}=b$. Without loss of generality,  we assume that $g_{2}=1, h_{2}=b^{q^k}$ and $g_{4}=a^{p^l}, h_{4}=1$. Then as  $x_{1}\sim x_{2}$ and $G$ is cyclic so $x_{1}\sim x_{3}$. Similarly, if we reverse $(g_{2}, h_{2})$ and $(g_{4}, h_{4})$ then we have either $x_{1}\sim x_{3}$ or $x_{3}\sim x_{5}$. Thus, $P(G)$ does not contain $P_{5}$. \\
Next we prove that $P(G)$ is $\overline{P_{5}}$-free. Clearly, if a graph contains $\overline{P_{5}}$ then the graph must have an induced $4$-vertex cycle. Suppose, $P(G)$ carries $\overline{P_{5}}$. Then $P(G)$ has a $4$-vertex induced cycle say $A\sim B\sim C\sim D\sim A$ with $C=(a, b)$. If one of $a$ or $b$ is $1$ then $B\sim D$. On the other hand, if $a, b$ are both power of $q$ then also $B\sim D$; whereas for the case $o(a)=q^{k}, o(b)=p^{l}$ we obtain $A\sim C$. Thus, $P(G)$ does not contain any $4$-vertex induced cycle and hence $P(G)$ is $\overline{P_{5}}$-free.
\end{proof}

\begin{theorem}
\label{S_{n}}
$P(S_{n})$ is  $\{P_{5}, \overline{P_{5}}\}$-free if and only if $n \leq 5$.
\end{theorem}

\begin{proof}
If $n \geq 6$ then $P(S_{n})$ contains a path $(5~6)\sim (1~2 ~3)(5~6) \sim (1~2~3)\sim (1~2~3)(4~5)\sim (4~5)\sim (4~5)(1~2~6)$. So, $n \leq 5$.\\
If $n \leq 5$ then the maximal cyclic subgroups of  $P(S_{n})$ intersect in the identity, and their orders are in the set $\{2\}$ (for $n=2$), $\{2, 3\}$ (for $n=3$), $\{2, 3, 4\}$ (for $n=4$), or $\{4, 5, 6\}$ (for $n=5$). As for each $n\leq 5$, the power graph of the maximal cyclic subgroups of $S_{n}$ does not contain any path of length $3$ and above so $P(S_{n})$ is $P_{5}$-free and $\overline{P_{5}}$-free (since, $\overline{P_{5}}$ contains induced $P_{4}$).
\end{proof}

\subsection{Simple Groups of Lie type and Sporadic smple group}
\begin{theorem}
If $G$ is a sporadic simple group then $P(G)$ is never $\{P_{5}, \overline{P_{5}}\}$-free.
\end{theorem}

\begin{proof}
Firstly consider the Mathieu group $M_{11}$. It contains $165$ elements of order $2$, $440$ elements of order $3$, $1320$ elements of order $6$ and $990$ elements of order $4$. So, there exist elements $a, b, c, d, e$ of orders $4, 2, 6, 3, 6$ resp. such that $a^{2}=b=c^{3}, c^{2}=d=e^{2}$ with $c^{3}\neq e^{3}$. Thus, $P(M_{11})$ contains $P_{5}$.\\
Since, $M_{11}$ is contained as a subgroup in every sporadic group except $J_{1}, J_{2}, J_{3}, M_{22}, He, Ru$ and $Th$ so their power graph contains $P_{5}$. \\
But $M_{22}$ contains $A_{7}$, $J_{1}, J_{2}, J_{3}$ contain $D_{3}\times D_{5}, A_{4}\times A_{5}, C_{3}\times A_{6}$ respectively, $He, Ru$ contain $S_{7}, A_{8}$ respectively. By Theorems \ref{S_{n}}, \ref{A_{n}}, \ref{th_p5_direct_product} the power graphs of these subgroups contains either $P_{5}$ or its complement. Hence, the power graphs of these groups are not $\{P_{5}, \overline{P_{5}}\}$-free.
From the information in $\mathbb{ATLAS}$ \cite{Conway}, we observe that $Th$ contains elements $x, y, z, w, u$ whose orders are $3, 6, 2, 10, 5$ respectively. Additionally these elements satify the conditions $y^{2}=x, y^{3}=z=w^{5}$ and $u=w^{2}$. Thus, $P(Th)$ contains a $5$ vertex induced path $x\sim y\sim z\sim w\sim u$. \\
This completes the proof of the theorem.
\end{proof}

\begin{theorem}
Let $G$ be a finite simple group of Lie type except the Ree group ${}^{2}{G_{2}}(q)$ (where, $q=3^{2e+1}$). Then $P(G)$ is $\{P_{5}, \overline{P_{5}}\}$-free if and only if either of the followings hold:\\
I) $G\cong A_{n}$ with $n \leq 6$;\\
II) $G\cong PSL(2, q)$ such that conditions a) or b) occurs:
a) the numbers $(q\pm 1)/2$ are either a prime power or product of some prime and a prime power if $q$ odd;\\
b) $q\pm 1$ are either a prime or product of some prime and a prime power if $q$ even;\\
III) $G={}^{2}{B_{2}}(q)=Sz(q)$, where $q=2^{2e+1}$ with the numbers $q-1, q\pm\sqrt{2q}+1$ are  either a prime power or product of some prime and a prime power;\\
IV) $G\cong PSL(3, 4)$.
\end{theorem}
We prove this theorem by proving the following subsequent theorems.
\medskip
\begin{theorem}
\label{A_{n}}
$P(A_{n})$ is $\{P_{5}, \overline{P_{5}}\}$-free if and only if $n \leq 6$.
\end{theorem}

\begin{proof}
For $n \geq 7$, $P(A_{n})$ contains a path $(1~2~3~4)(5~6)\sim (1~3)(2~4)\sim (1~3)(2~4)(5~6~7)\sim (5~6~7)\sim (1~2)(3~4)(5~6~7)$. Thus, $n \leq 6$.\\
For $n=4, 5, 6$ then prime graph of $A_{n}$ is a null graph; so $P(A_{n})$ must be $P_{5}$-free. Otherwise,  $P(A_{n})$ contains a $4$-vertex induced path which contradicts that $P(A_{n})$ is a cograph (see Theorem \ref{pre_th_2}). Since, $\overline{P_{5}}$ contains induced path $P_{4}$ and $P(A_{n})$ (where $n=4, 5, 6$) is $P_{4}$-free so $P(A_{n})$ is also $\overline{P_{5}}$-free. \\
If $n=3$ then $P(A_{3})$ is a complete graph (as $A_{3}$ is a cyclic group of order $3$). This implies that $P(A_{3})$ is $\{P_{5}, \overline{P_{5}}\}$-free.
\end{proof}

\begin{theorem}
Let $G\cong PSL(2, q)$. Then $P(G)$ is $\{P_{5}, \overline{P_{5}}\}$-free if and only if the followings hold:\\
a) the numbers $(q\pm 1)/2$ are either a prime power or product of some prime and a prime power if $q$ odd;\\
b) $q\pm 1$ are either a prime power or product of some prime and a prime power if $q$ even.
\end{theorem}

\begin{proof}
Let $q$ be a power of some odd prime. Now, $P^{*}(G)$ is the disjoint union of $P^{*}(C_{(q\pm1)/2})$ along with some isolated vertices. Thus if $P(G)$ is $\{P_{5}, \overline{P_{5}}\}$-free then  $P^{*}(C_{(q\pm1)/2})$ is also $\{P_{5}, \overline{P_{5}}\}$-free. This implies the condition in a) according to the Theorem \ref{th_p5_1}.\\
Suppose $q$ is a power of $2$. Then $P^*(G)$ is the disjoint union of $P^{*}(C_{(q\pm1)})$ along with some isolated vertices. Then by Theorem \ref{th_p5_1} the numbers $q\pm 1$ satisfy the conditions in b).
\end{proof}

\begin{theorem}
Let $G={}^{2}{B_{2}}(q)=Sz(q)$, where $q=2^{2e+1}$. Then  $P(G)$ is $\{P_{5}, \overline{P_{5}}\}$-free if and only if the numbers $q-1, q\pm\sqrt{2q}+1$ are  either a prime power or product of some prime and a prime power.
\end{theorem}

\begin{proof}
Here $G$ has $4$ maximal cyclic subgroups of orders $4, q-1,  q\pm\sqrt{2q}+1$. Since these $4$ numbers are pairwise coprime so any edge in $P(G)$ must lie in a maximal cyclic subgroup. Thus if $P(G)$ contains either $P_{5}$ and its complement then it must be contained in a maximal cyclic subgroup. Now the power graph of a cyclic group of order $4$ is a complete graph. Therefore, 
 $P(G)$ is $\{P_{5}, \overline{P_{5}}\}$-free if and only if the numbers $q-1, q\pm\sqrt{2q}+1$ satify the stated condition according to the Theorem \ref{th_p5_1}.
\end{proof}

\begin{theorem}
If $q (\geq 4)$ is a power of $2$ then $P(PSU(3, q))$ is never $\{P_{5}, \overline{P_{5}}\}$-free.
\end{theorem}

\begin{proof}
Let $\beta$ be a generator of the multiplicative group of $GF(q^{2})$. So $o(\beta)=q+1$. Let $p (>3)$ be a prime factor of $q+1$. Set $d=(q+1)/p$. Now $\alpha=\beta^{d(q-1)}$ has order $p$. Then $\overline{\alpha}=\beta^{d(q^{2}-1)}$ and $\alpha\overline{\alpha}=1$ in $GF(q^{2})$. Choose $3$ matrices $g, h, k$ as follows:\\
$$ 
g=\begin{bmatrix}
0 & 1& 0 \\
1& 0 & 0 \\
0& 0& 1
\end{bmatrix},~~h=\begin{bmatrix}
\alpha & 0& 0 \\
0 & \alpha & 0 \\
0& 0& \alpha^{-2}
\end{bmatrix}~~\text{and} ~~k=\begin{bmatrix}
0 & \alpha & 0 \\
1& 0 & 0 \\
0& 0& \alpha^{-1}
\end{bmatrix}
$$
 Then $o(g)=2, o(h)=p, o(k)=o(gh)=o(hk)=2p$ and $k^{2}=h$, $gh=hg$, $hk=kh$. So, $P(SU(3, q))$ contains the induced path $g\sim gh\sim h\sim k\sim k^{p}$. Since $g, h, k\in SU(3, q)\setminus Z$ so choosing $a=gZ, b=hZ, c=kZ$; then an induced path $a\sim ab \sim b \sim c \sim c^{p}$ is contained in $P(PSU(3, q))$. \\
But this argument is not valid when $q=8$. In that case, $P(PSU(3, 8))$ contains a subgroup $C_{3}\times PSL(2, 8)$ whose power graph is not $\{P_{5}, \overline{P_{5}}\}$-free (by Theorem \ref{th_p5_direct_product}).
\end{proof}

\begin{theorem}
If $q$ is a power of an odd prime then $P(PSU(3, q))$ is never $\{P_{5}, \overline{P_{5}}\}$-free.
\end{theorem}

\begin{proof}
If $q$ is odd, then $PSU(3, q)$ contains a cyclic subgroup of order $(q^{2}-1)/gcd(q+
1, 3)$. Since $q$ is odd, both the numbers $q -1$ and $(q+1)/gcd(q +1, 3)$ are even. Thus
$P(PSU(3, q))$ is $\{P_{5}, \overline{P_{5}}\}$-free  if $(q - 1)(q + 1)/gcd(q + 1, 3)$ is either a power of $2$ or of the
form $2^{k}p'$, where $p'$ is an odd prime.\\
First suppose that both $(q + 1)/gcd(q + 1, 3)$ and $q - 1$ are powers of $2$. As only
one of $q + 1$ and $q - 1$ is divisible by $4$, so the pair $(q - 1, q + 1)$  is either $(2, 4)$ or
$(4, 6)$. Hence $q = 3$ or $5$.\\
Next, suppose that $(q - 1)(q + 1)/gcd(q + 1, 3) = 2^{k}p'$. Then one of $q - 1, (q +
1)/gcd(q + 1, 3)$ is a power of $2$. Without loss of generality, we assume that $q - 1$
is a power of $2$. Now if $q \neq 3, 9$, then $q + 1$ is either of the forms $2p'$ or $6p'$ for
some odd prime $p'$. If $q + 1 = 6p'$, then $PSU(3, q)$ contains a subgroup $C_{2^k}\times C_{2p}$ whose power graph is not $\{P_{5}, \overline{P_{5}}\}$-free (by Theorem \ref{th_p5_1}). Again, if $q + 1 = 2p'$, then $PSU(3, q)$ contains the subgroup $C_{q+1}\times C_{q+1/gcd(q+1,3)}$. By Theorem \ref{th_p5_1}, $P(C_{2p'} \times C_{2p'})$  is not $\{P_{5}, \overline{P_{5}}\}$-free.
So either $q = 3$ or $9$ in this case. \\
But if $q=9$ then  $C_{8} \times C_{10}$ is contained in $PSU(3, q)$. For $q=5$, $PSU(3, 5)$ contains $A_{7}$. The power graph of none of these subgroups are $\{P_{5}, \overline{P_{5}}\}$-free [see Theorem \ref{th_p5_1}]. On the other hand, if $q=3$ then $PSU(3, 3)$ contains $4\cdot S_{4}$. Now, $4\cdot S_{4}$ is given by $⟨a, b, c, d, e|a^4 = d^3 = 1, b^2 =
c^2 = e^2 = a^2, ab = ba, ac = ca, ad = da, eae^{-1} = a^{-1}, cbc^{-1} = a^{2}b, dbd^{-1} = a^{2}bc, ebe^{-1} =
bc, dcd^{-1} = b, ece^{-1} = a^{2}c, ede^{-1} = d^{-1}⟩$. Then $P(PSU(3, 5))$ contains the induced path $a\sim ad\sim d\sim ed\sim e$. Thus, in any of the case $q=3, 5, 9$, the power graph of $PSU(3, q)$ is not $\{P_{5}, \overline{P_{5}}\}$-free.
\end{proof}

\begin{remark}
Let $G$ be the Ree group ${}^{2}{G_{2}}(q) = R_1(q)$, where $q = 3^{2e+1}$. We observe that $C_2 \times PSL(2, q)$ is the centralizer of an involution in the
group $G$. The group $C_2 \times PSL(2, q)$ (see \cite{wilson}) carries the subgroups $C_2 \times C_{(q\pm1)/2}$. Therefore, by
Theorem \ref{th_p5_1}, $P(G)$  is $\{P_{5}, \overline{P_{5}}\}$-free  if $(q \pm1)/2$  is either a power of $2$
or a power of an odd prime or of the form $2p^r$.\\
If both $q \pm 1$ are powers of $2$, then we get a solution of Catalan’s conjecture
which contradicts Mihailescu’s theorem (see \cite[Section 6.11]{Cohn}).\\
Let $q + 1 = 2p_{1}^{r_{1}}$ and $q - 1 = 2p_{2}^{r_2}$, where $r_1$ and $r_2$ are odd primes. Then the
diophantine equation $p_{1}^{r_1}-p_{2}^{r_2}=1$ has a solution. This leads to a contradiction
to Mihailescu’s theorem as both $p_1$ and $p_2$ are odd.\\
Similarly, both $q \pm 1$ can not be of the form $4p^{k}$ as the diophantine equation $2x - 2y = 1$ has no solution.\\
Again, if any one of $q + 1$ or $q -1$ is $4p_{1}^{s}$ and the other one is $2p_{2}^{t}$ then, for $x = p_{1}^{s}$ and
$y = p_{2}^{t}$, the corresponding diophantine equation is either $2x - y = 1$ or $x - 2y = 1$. 
\end{remark}
One can check that the solutions exist for infinitely many values of $x, y$. So in this case, the question arise :\\
\textbf{Problem 1:} Does there exist infinitely many values of $q$ for which $P(G)$ is $\{P_{5}, \overline{P_{5}}\}$-free?

\begin{theorem}
Let $q$ be power of an odd prime. Then $P(PSL(3, q))$ is never $\{P_{5}, \overline{P_{5}}\}$-free.
\end{theorem}

\begin{proof}
Consider $3$ matrices $g, h, x$ in $SL(3, q)$ as follows:\\
$$
g=\begin{bmatrix}
0 & -1 & 0\\
1 & 0 & 0\\
0 & 0 & 1
\end{bmatrix},~~h=\begin{bmatrix}
0 & -1 & 0\\
1 & 1 & 0\\
0 & 0 & 1
\end{bmatrix}~~\text{and}~~x=\begin{bmatrix}
0 & 1 & 0\\
-1 & -1 & 0\\
0 & 0 & -1
\end{bmatrix}
$$
Here, $h^{2}=x^{2}, h^{3}=g^{2}$ and $o(h)=o(x)=6, o(g)=4$. Then the induced path $x \sim h^{2}\sim h\sim g^{2}\sim g$ is contained in $P(SL(3, q))$. Since, $g, h, x \in SL(3, q)\setminus Z$ so set $a=gZ, b=hZ, c=xZ$ and $P(PSL(3, q))$ contains the path $c\sim b^{2}\sim b \sim a^{2}\sim a$. This completes the proof of the theorem.
\end{proof}

\begin{theorem}
$P(PSL(3, q))$ (where, $q (\geq 4)$ is a power of $2$) is $\{P_{5}, \overline{P_{5}}\}$-free if $q=2, 4$.
\end{theorem}

\begin{proof}
Firstly, let $q$ be a power of $2$ with $q>4$. If $q$ is an odd power of $2$ then $q-1$ is not divisible by $3$; whereas if $q$ is an even power of $2$ then $q-1$ is not a power of $3$ (by the solution of Catalan's conjecture (see \cite[Section 6.11]{Cohn}). In that case $q-1$ must have a large prime divisor. Let $\alpha$ be an element in the multiplicative group of $GF(q)$ such that $o(\alpha)=p>3$. Choose the matrices $h~\text{and}~ k$ as follows:\\
$$
h=\begin{bmatrix}
\alpha & 0 & 0\\
0 & \alpha & 0\\
0 & 0 & \alpha^{-2}
\end{bmatrix}~~\text{and}~~k=\begin{bmatrix}
0 & \alpha & 0\\
1 & 0 & 0\\
0 & 0 & \alpha^{-1}
\end{bmatrix}
$$
Set $x=-k$. Then $k^{2}=h=x^{2}, o(k)=o(x)=2p, o(h)=p$ and $o(hk^{p})=2p$; hence $P(SL(3, q))$ contains a $5$-vertex induced path $k^{p}\sim hk^{p}\sim h \sim x \sim x^{p}$. Since the matrices are not scalar so $P(PSL(3, q))$ carries the induced path $P_{5}$. Now the remaining cases are $q=2, 4$.\\
If $q=2, 4$ then the power graphs of $PSL(3, 2), PSL(3, 4)$ have prime graph which is a null graph so their power graphs are cograph (by Theorem \ref{pre_th_2}) so they are also $\{P_{5}, \overline{P_{5}}\}$-free (as $\overline{P_{5}}, P_{5}$ has $P_{4}$ as an induced subgraph).
\end{proof}

\begin{theorem}
$P(PSp(4, q))$ is never $\{P_{5}, \overline{P_{5}}\}$-free.
\end{theorem}

\begin{proof}
 First, suppose that $q$ is a power of $2$. Then $PSp(4, q)$ contains $PSL(2, q)\times PSL(2, q)$, and so it contains $C_{(q\pm1)} \times C_{(q\pm1)}$. Now $(q +1, q -1) = 1$ and $3$ divides one of them. Thus, $P(G)$ is $\{P_{5}, \overline{P_{5}}\}$-free  if and only if one of $q \pm 1$ is a prime and the other is a power of another prime. So we must have $q = 2, 4$ or $8$.
Next, suppose that $q$ is a power of an odd prime. Then $G$ contains the central
product of two copies of $SL(2, q)$, and hence it contains $C_{(q\pm1)} \times C_{(q\pm1)/2}$. Thus if $PSp(4, q)$  is $\{P_{5}, \overline{P_{5}}\}$-free, then both $(q \pm 1)/2$ have to be prime powers (as
$C_{(q\pm1)}\times C_{(q\pm1)/2}$ is contained in $PSp(4, q))$. Now one of $(q -1)/2$ or $(q +1)/2$  is even, so
one of $(q\pm1)/2$ must be a power of $2$. This implies one of $C_{(q\pm1)}$ must be $4$, or else $PSp(4, q)$ contains a subgroup whose power graph is $\{P_{5}, \overline{P_{5}}\}$. Thus the possible values of $q$ are $3$ and $5$.\\
If $q = 2$, then $PSp(4, 2)$  is isomorphic to $S_{6}$. By Theorem \ref{S_{n}}, $P(PSp(4, 2))$ is not $\{P_{5}, \overline{P_{5}}\}$-free.\\
If $q = 3$ or $4$, then $PSp(4, q)$ contains $S_6$, and so its power graph is not $\{P_{5}, \overline{P_{5}}\}$-free. \\
Again, $PSp(4, 5)$ contains the subgroup $S_{3}\times S_{5}$ whose power graph is not $\{P_{5}, \overline{P_{5}}\}$-free.\\
graph. The group $PSp(4, 8)$ contains $PSp(4, 2)$ (see Mitchell Theorem \cite{King}), and so
$P(PSp(4, 8))$ is not $\{P_{5}, \overline{P_{5}}\}$-free. Thus, we get our conclusion.
\end{proof}

\begin{theorem}
The power graph of $G_{2}(q)$  is never $\{P_{5}, \overline{P_{5}}\}$-free.
\end{theorem}
\begin{proof}
$G_{2}(q)$ comprises $SL(3, q)$ and $SU(3, q)$ as the subgroups (see \cite{cooperstein, kleidman2}). Now $SL(3, q) \cong PSL(3, q)$  if $q \not \equiv 1(mod~3)$ and $SU(3, q)\cong PSU(3, q)$  if $q \not \equiv -1(mod~3)$. Thus, for any $q$, either $PSL(3, q)$ or $PSU(3, q)$ is contained in $G_{2}(q)$. Now, $PSL(3, q)$ is $\{P_{5}, \overline{P_{5}}\}$-free if $q = 2, 4$, whereas $PSU(3, q)$ is not $\{P_{5}, \overline{P_{5}}\}$-free for every $q \neq 2$. So only remaining case is $q = 2$. The group $G_{2}(2)$ is not simple,
and it contains $PSU(3, 3)$ as a subgroup. Hence, the power graph of $G_{2}(q)$ is not $\{P_{5}, \overline{P_{5}}\}$-free. So we arrive at the conclusion.
\end{proof}
We now consider the other simple groups of Lie type of rank $2$.\\
1) Let $G = PSU(4, q)$. If $q > 2$, then $G$ contains $PSp(4, q)$.
Again, if $q = 2$, then $G$ isomorphic to $PSp(4, 3)$. Hence, $P(PSU(4, q))$ is not $\{P_{5}, \overline{P_{5}}\}$-free.\\
2) Let $G = PSU(5, q)$. If $q = 2$, then $PSU(5, q)$ contains $PSU(4, 2)$ as a subgroup. Again, if $q > 2$, then $PSU(5, q)$ contains $SL(2, q)\times SL(2, q)$  (see \cite{wilson}). Now, by Theorem \ref{th_p5_direct_product}, $P(SL(2, q)\times SL(2, q))$ is not $\{P_{5}, \overline{P_{5}}\}$-free, and so $P(PSU(5, q))$ is not $\{P_{5}, \overline{P_{5}}\}$-free. \\
3) The group ${}^{2}{F_{4}}(2^d)$ contains ${}^{2}{F_{4}}(2)$ for all odd $d$  (see \cite{malle}), and so it contains $PSL(3, 3)$.
Thus, ${}^{2}{F_{4}}(2^d)$ is not $\{P_{5}, \overline{P_{5}}\}$-free.\\
4) The group ${}^{3}{D_{4}}(q)$ contains $G_{2}(q)$ (see \cite{kleidman}); so its power graph is also a not $\{P_{5}, \overline{P_{5}}\}$-free.

\begin{theorem}
Let $G$ be a Lie type simple groups of rank more than $2$. Then $P(G)$ is never $\{P_{5}, \overline{P_{5}}\}$-free.
\end{theorem}

\begin{proof}
Let $G$ be a simple group of Lie type of rank greater than $2$. Now, as the Dynkin
diagram of G has a single bond in each case, $G$ contains $PSL(3, q)$ as a subgroup. But the
only values for which the power graph of $PSL(3, q)$  is $\{P_{5}, \overline{P_{5}}\}$-free are $q = 2, 4$. Hence, we have
to check only for those simple groups whose underlying fields are finite fields with $2$ and $4$ elements, respectively.
Now $PSL(4, 2)$ is isomorphic to $A_8$. So its power graph is not $\{P_{5}, \overline{P_{5}}\}$-free. The group $PSL(4, 4)$ contains $PSL(4, 2)$, whose power graph is not $\{P_{5}, \overline{P_{5}}\}$-free.\\
On the other hand, $PSp(6, 2)$ contains $S_8$, whose power graph is not $\{P_{5}, \overline{P_{5}}\}$-free. Again, $PSp(6, 4)$ contains $PSp(6, 2)$. So, in this case, we get both $PSp(6, 2), PSp(6, 4)$ whose power graphs are not  $\{P_{5}, \overline{P_{5}}\}$-free.\\
If $q = 2, 4$, then the orthogonal and unitary groups of Lie type of rank $3$
contain $PSp(4, q)$. Hence their power graphs are not $\{P_{5}, \overline{P_{5}}\}$-free.
\end{proof}

\section{$\{P_{2}\cup P_{3}, \overline{P_{2}\cup P_{3}}\}$-free}

\begin{theorem}
For any finite nilpotent group $G$, $P(G)$ is $\{P_{2}\cup P_{3}, \overline{P_{2}\cup P_{3}}\}$-free if and only if $G$ is either a $p$-group or a cyclic group $C_{p^{a}q}$ (where $a \geq 1$) or $P\times C_{q^b}$, where $P$ is a non-cyclic $2$-group of exponent $2$ and $b \geq 1$.
\end{theorem}

\begin{proof}
Let $G$ be a finite nilpotent group such that $P(G)$ is $\{P_{2}\cup P_{3}, \overline{P_{2}\cup P_{3}}\}$-free. \\
 We claim that $o(G)$ must have at most $2$ distinct prime divisors. If possible let $o(G)$ be divisible by $3$ distinct primes say $p<q< r$ with the corresponding elements $a, b, c$. Then the vertices $\{c, c^{-1}\}\cup \{a, ab, b\}$ form $P_{2}\cup P_{3}$.Thus $o(G)$ is either a prime power or of the form $p^{a}q^{b}$(where, $a, b \geq 1$).\\
Let $o(G)=p^{a}q^{b}$ (with $p<q$) and $P, Q$ be the Sylow $p$- and Sylow $q$-subgroups of $G$. Here we consider two cases based on $p$.\\
\textbf{Case 1.} $p$ odd \\
If any one of two Sylow subgroups is non-cyclic, say $P$, then $P(G)$ contains $P_{2}\cup P_{3}$ by the vertices $\{a, a^{-1}\}\cup \{b, bc, c\}$ (where, $o(a)=o(b)=p, o(c)=q, a\nsim b$). Therefore in this case $G\cong C_{p^{a}q^{b}}$. But if both $a, b>1$ then $P(G)$ contains $\overline{P_{2}\cup P_{3}}$. Thus in this case $G\cong C_{p^{a}q}$.\\
\textbf{Case 2.} $p=2$\\
Clearly, the Sylow $q$-subgroup must be cyclic; otherwise we get $P_{2}\cup P_{3}$ in $P(G)$.\\
If the Sylow $2$-subgroup $P$ of $G$ is cyclic then $G\cong C_{2^{a}q^{b}}$, where $a, b \geq 1$. Again, we have either $a$ or $b$ must be $1$; otherwise $P(G)$ contains $\overline{P_{2}\cup P_{3}}$. Thus $G\cong C_{2^{a}q}$. \\
Let the Sylow $2$-subgroup $P$ of $G$ be non-cyclic. If $P$ has two distinct elements of order power of $2$ that are non-adjacent in $P(G)$ then $P(G)$ carries $P_{2}\cup P_{3}$. Thus, either $P$ is cyclic or a non-cyclic $2$-group of exponent $2$.\\
 Therefore $G$ is either a $p$-group or a cyclic group $C_{p^{a}q}$ with $a \geq 1$ or $P\times C_{q^b}$, where $b \geq 1$ and $P$ is a non-cyclic $2$-group of exponent $2$.\\
\underline{Converse Part:}\\
Let $G$ be a $p$-group. Then any $3$ adjacent vertices form a triangle in $P(G)$. Thus $P(G)$ is $\{P_{2}\cup P_{3}, \overline{P_{2}\cup P_{3}}\}$-free. \\
If $G \cong C_{p^{a}q}$ (where $a, b \geq 1$), then as $G$ has unique subgroup of each orders dividing $o(G)$, so it is easy to confirm that $P(G)$ is $\{P_{2}\cup P_{3}\}$-free. On the other hand, if $P(G)$ contains $\overline{P_{2}\cup P_{3}}$ then $P(G)$ has induced $4$-cycle. But since $P(G)$ is chordal by Theorem \ref{pre_th_3}, so such $C_{4}$ never exists. Hence $P(G)$ is  $\{P_{2}\cup P_{3}, \overline{P_{2}\cup P_{3}}\}$-free. \\
Next, suppose $G \cong P\times C_{q^b}$, where $P$ is a non-cyclic $2$-group of exponent $2$ and $b \geq 1$. Then it is easy to check that $P(G)$ is $P_{2}\cup P_{3}$-free. But suppose $P(G)$ contains $\overline{P_{2}\cup P_{3}}$. Then $P(G)$ comprises a $4$-vertex induced cycle, which contradicts the Theorem \ref{pre_th_3} that is $P(G)$ is a chordal graph. Hence $P(G)$ is $\{P_{2}\cup P_{3}, \overline{P_{2}\cup P_{3}}\}$-free in this case.
\end{proof}

\begin{theorem}
Let $G$ be a finite non-nilpotent group. Then $P(G)$ is $\{P_{2}\cup P_{3}, \overline{P_{2}\cup P_{3}}\}$-free if and only if $G$ has one of the following possibilities:\\
(a) $|\pi(G)|\geq 4$ and $G$ is an EPPO group;\\
(b) for $|\pi(G)|=3$, $G$ is either an EPPO group or the group $C_{q^{a}r}\rtimes P$ of order $2^{a}q^{b}r^{c}$, where $2<q<r$ with $P, Q, R$ are the Sylow $2$-, Sylow $q$- and Sylow $r$-subgroups of $G$ respectively and $P$ must be a $2$-group of exponent $2$;\\
(c) if $\pi(G)=\{p, q\}$ then $G$ is either an EPPO group or a group of order $2^{a}q^{b}$ ($q$ odd) such that  the Sylow $q$-subgroup of $G$ must be cyclic as well as normal and the Sylow $2$-subgroups of $G$ must be of exponent $2$.
\end{theorem}

\begin{proof}
Let $G$ be a non-nilpotent group such that $P(G)$ is $\{P_{2}\cup P_{3}, \overline{P_{2}\cup P_{3}}\}$-free.\\
If $o(G)$ has $4$ and more divisors then we claim that $G$ must be an EPPO group. Otherwise, $o(G)$ has at least $3$ distinct odd prime divisors say $p, q, r$. Since $G$ is non-EPPO so contains an element whose order is product of two distinct primes. Thus $P(G)$ contains $P_{3}$. Without loss of generality, we suppose that $a\sim b \sim c$ are $3$ vertices of $P_{3}$ with orders $p, pq, q$. There exists an element in $G$ of order $r$, say $c$ and then the pair $\{c, c^{-1}\}$ form $P_{2}$. This implies that $P(G)$ contains $\{P_{2}\cup P_{3}\}$. Thus, in this case $G$ must be an EPPO group.\\
Suppose $o(G)$ has $3$ distinct prime divisors say $p<q<r$.\\
 Here we consider two cases depending on $p$:\\
\textbf{Case 1.} $p$ is odd prime \\
Here we claim that $G$ must be an EPPO group; otherwise $G$ contains an element of order product of two distinct primes. As $p, q, r$ all are odd primes so $P(G)$ has $\{P_{2}\cup P_{3}\}$. \\
\textbf{Case 2.} $p=2$\\
In this case $G$ may be an EPPO group. \\
Now let $G$ be a non-EPPO group. Obviously, $G$ cannot contain any element of order $2q$ or $2r$. In that case if $G$ has an element say $x$ of order $2q$ then $P(G)$ carries the path $x^{q}, x, x^{2}$ along with the $P_{2}$ by the vertices $\{c, c^{-1}\}$, where $o(c)=r$. Thus, $G$ can contains an element of order $qr$ only. Additionally, we observe that the Sylow $q$-, Sylow $r$-subgroup must be cyclic as well as normal and Sylow $2$-subgroup must be of exponent $2$. Elsewhere $\{P_{2}\cup P_{3}\}$ is contained in $P(G)$. Thus $G$ has a normal subgroup $C_{q^{a}r^{b}}$ (as Sylow $q$-and $r$-subgroups are normal so their product is also normal) and all the elements outside of the normal subgroup  must be of order $2$. Moreover, we observe that $P(C_{q^{a}r^{b}})$ carries $\overline{P_{2}\cup P_{3}}$ if both $a, b>1$. So any one of two Sylow $q$-subgroup or Sylow $r$-subgroup must be of prime order. Let $Q\cong C_{q^a}, R\cong C_{r}$. Therefore, $G$ is the group $C_{q^{a}r}\rtimes P$ where $P, Q, R$ are the Sylow $2$-, Sylow $q$- and Sylow $r$-subgroups of $G$ respectively and $P$ must be a $2$-group of exponent $2$. \\
Suppose that $o(G)$ has exactly two distinct prime divisors say $p, q$. If both $p, q$ are odd primes then $G$ must be an EPPO group; elsewhere both the Sylow subgroups of $G$ must be cyclic as well as normal which forces $G$ to be a nilpotent group. This leads to a contradiction as $G$ is not nilpotent. Now consider $o(G)=2^{k}q^{m}$. If $G$ is a non-EPPO group then the Sylow $q$-subgroup of $G$ must be cyclic as well as normal as otherwise $P(G)$ comprises $\{P_{2}\cup P_{3}\}$. Moreover, the Sylow $2$-subgroups of $G$ must be of exponent $2$ since $P(G)$ is $\{P_{2}\cup P_{3}\}$-free. \\
\underline{Converse Part:}\\
(1) Let $G$ be an EPPO group. Then any $3$ consecutive adjacent vertices must belong to the same cyclic subgroup of prime power order. So they form a triangle. Hence, $P(G)$ is $\{P_{2}\cup P_{3}, \overline{P_{2}\cup P_{3}}\}$-free. \\
(2) Let $G\cong C_{q^{a}r}\rtimes P$ along with the prescribed condition in (b). Clearly $P(G)$ never contains $P_{2}\cup P_{3}$. Every elements (non-identity) are of order power of $q$, or $r$, or $2$, or of the form $q^{i}r$. If $P(G)$ contains $\overline{P_{2}\cup P_{3}}$ then an induced $4$-cycle is contained in $P(G)$. So, to prove that $P(G)$ is $\overline{P_{2}\cup P_{3}}$-free it is enough to show that $P(G)$ is induced cycle $C_{4}$-free. Clearly in  all the vertices of this cycle must belong to the cyclic subgroup $C_{q^{a}r}$. As, $P(C_{q^{a}r})$ is a chordal graph by Theorem \ref{pre_th_3} so $P(G)$ is $C_{4}$-free. Hence, $P(G)$ is $\{P_{2}\cup P_{3}, \overline{P_{2}\cup P_{3}}\}$-free.\\
(3) If $G$ is a group of order $2^{a}q^{b}$ ($q$ odd) such that the conditions in (c) hold. One can easily observe that $P(G)$ is $P_{2}\cup P_{3}$-free. Also, using the same argument as done in (2) we can conclude that $P(G)$ is $\overline{P_{2}\cup P_{3}}$-free. Hence, in this case $P(G)$ is $\{P_{2}\cup P_{3}, \overline{P_{2}\cup P_{3}}\}$-free.
\end{proof}

\section{Diamond-free}
The diamond graph is a planar, undirected and simple graph with $4$ vertices and $5$ edges. It consists of a complete graph $K_{4}$ with an edge deletion. It looks like:
\vspace*{3mm}

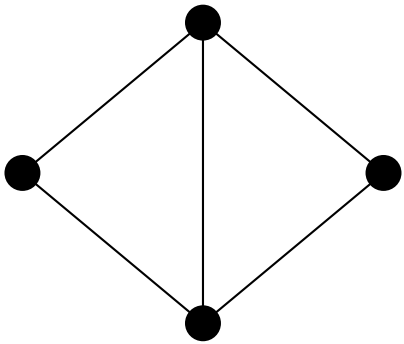
\begin{figure}[ht!]
\begin{center}
\begin{tikzpicture}[scale=0.8,style=thick]
\def\vr{8pt} 
\draw (0,0)  [fill=black] circle (\vr);
\draw (6,0)  [fill=black] circle (\vr);
\draw (3,2.5)  [fill=black] circle (\vr);
\draw (3,-2.5)  [fill=black] circle (\vr);
\draw (0,0)-- (3,2.5);
\draw (0,0)-- (3,-2.5);
\draw (6,0)-- (3,2.5);
\draw (6,0)-- (3,-2.5);
\draw (3,2.5)-- (3,-2.5);
\end{tikzpicture}
\end{center}
\vspace{-0.3cm}
\caption{Diamond graph}
\label{fig: T-k}
\vspace{-0.7cm}
\end{figure}

 \vspace*{8mm}

The complement of a diamond graph is called a co-diamond graph. Here we identify the finite groups having diamond-free as well co-diamond free power graph.

\begin{lemma}
\label{lm_dia}
Let $G$ be a finite group. Then $P(G)$ is a diamond-free graph if and only if $G$ is either a $p$-group or an EPPO group.
\end{lemma}

\begin{proof}
First suppose, $G$ is a finite  group for which $P(G)$ is diamond-free.\\
If possible let, $o(G)$ have $2$ or more distinct prime divisors. Let $p, q$ be two distinct prime divisors of $o(G)$. We claim that $G$ must be an EPPO group. If not, then there exist at least two elements of order $pq$. We choose the consecutive $4$ vertices as $a, b, c, d$ of order $p, pq, q, pq$ with $d=b^{-1}$. Then they form a diamond in $P(G)$. This gives $G$ must be an EPPO group.\\
Therefore, $G$ is either a $p$-group or an EPPO group.\\
\underline{Converse Part:}\\
Let $G$ be a $p$-group. For the sake of contradiction, suppose $P(G)$ contains a diamond with the $4$ consecutive vertices as $a, b, c, d$. Clearly, as any $3$ consecutive adjacent vertices belong to same cyclic $p$-subgroup so they form a triangle in $P(G)$. This implies $a, b, c, d$ must be the complete graph $K_{4}$, which contradicts that $a, b, c, d$ is a diamond. \\
Again, let $G$ be an EPPO group. If possible let $P(G)$ contain a diamond where the $4$ consecutive vertices as $a, b, c, d$. Suppose, $degree(b)=degree(d)=3$ and $degree(a)=degree(c)=2$. Now as $a\sim b \sim c$ and $G$ is EPPO group so $a, b, c$ belong to same cyclic subgroup of prime power order. In that case $a$ must adjacent to $c$. This contradicts the $a, b, c, d$ form a diamond. \\
Thus in any cases $P(G)$ is diamond-free.
\end{proof}

\begin{theorem}
Let $G$ be a finite group. Then $P(G)$ is $\{even-hole, diamond\}$-free if and only if $G$ is either a $p$-group or an EPPO group.
\end{theorem}

\begin{proof}
Firstly, let $G$ be a group with $P(G)$ is $\{even-hole, diamond\}$-free. Since, $P(G)$ is diamond-free so $G$ is either a $p$-group or an EPPO group. \\
\underline{Converse Part:}\\
Let $G$ be either a $p$-group or an EPPO group. Then $P(G)$ is a chordal graph [see Theorem \ref{pre_th_3}]. So, $P(G)$ does not contain any even-hole. Also, by Lemma \ref{lm_dia} $P(G)$ is diamond-free. Hence  $P(G)$ is $\{even-hole, diamond\}$-free.
\end{proof}
In the next result we use the Kulakoff theorem from group theory which states that:
\begin{theorem}[Kulakoff Theorem]
Let $G$ be a p-group of order $p^\alpha$. Then
\begin{itemize}
\item[(a)]
the number of subgroup of prime power order is congruent to $1$ (mod $p$).
\item[(b)]
if $G$ has unique subgroup of order $p^{\beta}$ for all $\beta$ with $1<\beta \leq \alpha$, then $G$ is cyclic or $\beta =1$ and $p=2$, $G$ is the generalized quaternion group $Q_{2^{\alpha}}$.
\end{itemize} 
\end{theorem}

\begin{theorem}
Let $G$ be a finite group. Then $P(G)$ is $\{diamond, co-diamond\}$-free if and only if $G$ is either  a cyclic group of prime power order or a $2$-group of exponent $2$.
\end{theorem}

\begin{proof}
For any finite group $G$, let $P(G)$ be $\{diamond, co-diamond\}$-free. As, $P(G)$ is diamond-free so $G$ is either a $p$-group or an EPPO group. \\
But if $G$ is an EPPO group then $o(G)$ can have exactly two distinct prime divisors. Otherwise, $P(G)$ contains a co-diamond. Additionally, we observe that both the Sylow subgroup of $G$ must be cyclic as well as normal; elsewhere if Sylow $p$-subgroup is either non-cyclic or not normal then $G$ contains $3$ elements say $a, b$ of order $p$ with $a\nsim b$ in $P(G)$ and $c\in G$ of order $q$. Clearly, $a,b, c, c^{-1}$ form a co-diamond. Next, let $p=2$ such that the Sylow $2$-subgroup is $C_{2}$ and normal and Sylow $q$-subgroups are either non-cyclic or not normal. Then $G$ contains two elements $x, y$ of order $q$ with $x \nsim y$ and $a\in P$. So, $a, x, x^{-1}, y$ form a co-diamond in $P(G)$. Therefore, in any cases, $o(G)$ has exactly two distinct prime divisors with all Sylow subgroups are cyclic and normal. This implies $G$ must be the group $C_{p^{r}q^{s}}$, which contradicts that $G$ is an EPPO group. Hence $G$ must be a $p$-group. \\
If $G$ is a $p$-group we claim that $G$ is either a cyclic group of prime power order or a $2$-group of exponent $2$. \\
Let $p$ be an odd prime. If $G$ is non-cyclic then (by Kulakoff theorem) there exist at least $3$ distinct subgroups of order $p$. Thus $P(G)$ contains a co-diamond. Hence in this case $G$ must be a cyclic group of prime power order.\\
Now let $p=2$. If $G$ is cyclic then $P(G)$ is complete and so $\{diamond, co-diamond\}$-free. But if $G$ is non-cyclic then either $G$ is generalized quaternion group or $G$ has at least $3$ distinct minimal subgroups. For the latter case, $G$ does not contain any element of order $4$ and above, because otherwise we get a co-diamond. This gives $G$ must be a $2$-group of exponent $2$.\\
Next consider $G$ as generalized quaternion. Since in this case $G$ has at least $3$ distinct elements of order $4$ so $P(G)$ carries a co-diamond. Therefore, if $G$ is a $p$-group then $G$ is either a cyclic group of prime power order or a $2$-group of exponent $2$.\\
Converse part is obvious.
\end{proof}

\section{Acknowledgement}
The author Pallabi Manna is supported by Department of Atomic Energy (DAE), India and Santanu Mandal acknowledges VIT Bhopal University, India, for providing the infrastructure.

\section{Statements and Declarations} \textbf{Competing Interests:} The authors made no mention of any potential conflicts of interest.

\section{Data Availability}
Data sharing is not applicable to this article as no data were created or analyzed in this study.

\end{document}